\documentclass[a4paper,11pt,twoside]{paper}
\usepackage[T1]{fontenc}
\newif\ifRR\RRfalse
\RRtrue
\ifRR
  \usepackage{RR}
\fi
\usepackage{amsmath,amssymb,amsthm}
\usepackage{cite}

\topsep\medskipamount
\usepackage{parskip}
\paragraphfont{\sffamily\itshape}

\newtheoremstyle{theoremes}{\medskipamount}{\medskipamount}{\itshape}
                        {0pt}{\bfseries\sffamily}{.}{ }{}
\newtheoremstyle{remarques}{\medskipamount}{\medskipamount}{}
                        {0pt}{\bfseries\sffamily}{.}{ }{}
\theoremstyle{theoremes}

\newtheorem{theo}{Theorem}[section]
\newtheorem{cor}[theo]{Corollary}
\newtheorem{conj}[theo]{Conjecture}
\newtheorem{lem}[theo]{Lemma}

\theoremstyle{remarques}
\newtheorem{rem}{Remark}

\makeatletter
\@addtoreset{equation}{section}
\makeatother

\newcommand{\Rset}[1]{\mbox{$\mathbb{R}^{#1}$}}
\newcommand{\F}{\mbox{$\mathbb{F}$}}

\newcommand{\CA}{\mathcal{A}}
\newcommand{\CB}{\mathcal{B}}
\newcommand{\CC}{\mathcal{C}}
\newcommand{\CD}{\mathcal{D}}
\newcommand{\CF}{\mathcal{F}}
\newcommand{\CN}{\mathcal{N}}
\newcommand{\CO}{\mathcal{O}}

\newcommand{\CW}{\mathcal{W}}
\font\calcal=cmsy10 scaled\magstep1
\def\build#1_#2^#3{\mathrel{\mathop{\kern 0pt#1}\limits_{#2}^{#3}}}
\def\liml{\build{\longrightarrow}_{}^{{\mbox{\calcal L}}}}

\newcommand{\ep}{\mbox{$\varepsilon$}}
\newcommand\1{\leavevmode\hbox{\rm \small1\kern-0.35em\normalsize1}}

\newcommand\egaldef{\stackrel{\mbox{\upshape\tiny def}}{=}}
\DeclareMathOperator{\tr}{tr}
\DeclareMathOperator{\var}{var}
\def\E{\mbox{$\mathbb{E}$}}

\begin{document}
\ifRR

\RRdate{December 2008}
\RRauthor{Bernard Bercu \thanks{Université Bordeaux 1,
Institut de Mathématiques de Bordeaux UMR 5251, and INRIA Bordeaux Sud-Ouest, 351 cours de la Libération, 33405 Talence Cedex, France.  {\tt Bernard.Bercu@math.u-bordeaux1.fr}} \and 
Peggy Cénac \thanks{Université de Bourgogne, Institut de Mathématiques de Bourgogne, UMR 5584, 9 rue Alain Savary, BP 47870, 21078 Dijon Cedex, France.   {\tt Peggy.Cenac@u-bourgogne.fr}} \and
Guy Fayolle \thanks{INRIA CR Paris-Rocquencourt, Domaine de Voluceau, BP 105, 78153 Le Chesnay Cedex, France. {\tt Guy.Fayolle@inria.fr}}} 

\RRtitle{Théorème de la limite centrale pour les martingales vectorielles : convergence des moments et applications statistiques}

\RRetitle{On the Almost Sure Central Limit Theorem for Vector Martingales\,: Convergence of Moments and Statistical Applications}

\titlehead{Convergence of Moments in the ASCLT for Vector Martingales}

\RRresume{On étudie dans ce rapport des propriétés de convergence presque sûre de transformées de martingales vectorielles. Sous certaines conditions d'existence de moments et de régularité du processus croissant, on montre en particulier la convergence des moments normalisés de tout ordre pair dans le théorème central limite presque sûr pour les martingales vectorielles. On formule une conjecture de borne presque sûre, sous des hypothèses moins restrictives et couvrant des familles plus vastes de processus. Enfin, on applique ces résultats à des exemples issus d'applications statistiques, notamment les modèles autorégressifs linéaires et certains processus de branchement avec immigration, ce qui permet d'établir de nouvelles propriétés asymptotiques sur les erreurs d'estimation et
de prédiction.}

\RRabstract{We investigate the almost sure asymptotic properties of vector martingale transforms. Assuming some appropriate regularity conditions both on the increasing process and on the moments of the martingale, we prove that normalized moments of any even order converge in the almost sure cental limit theorem for martingales. A conjecture about almost sure upper bounds under wider hypotheses is formulated. The theoretical results are supported by examples borrowed from statistical applications, including linear autoregressive models and branching processes with immigration, for which new asymptotic properties are established on estimation and prediction errors.}

\RRmotcle{Théorème central limite presque sûr, martingale vectorielle, moment, regression stochastique.}

\RRkeyword{Almost sure central limit theorem, vector martingale, moment, stochastic regression.}

\RRprojet{Imara}

\RRtheme{\THNum}
\RRNo{6780}
\URRocq

\makeRR

\else
\title{On the Almost Sure Central Limit Theorem for Vector Martingales\,: Convergence of Moments and Statistical Applications} 

\author{Bernard Bercu \thanks{Université Bordeaux 1,
Institut de Mathématiques de Bordeaux UMR 5251, and INRIA Bordeaux Sud-Ouest, 351 cours de la Libération, 33405 Talence Cedex, France.  {\tt Bernard.Bercu@math.u-bordeaux1.fr}} \and 
Peggy Cénac \thanks{Université de Bourgogne, Institut de Mathématiques de Bourgogne, UMR 5584, 9 rue Alain Savary, BP 47870, 21078 Dijon Cedex, France.   {\tt Peggy.Cenac@u-bourgogne.fr}} \and
Guy Fayolle \thanks{INRIA CR Paris-Rocquencourt, Domaine de Voluceau, BP 105, 78153 Le Chesnay Cedex, France. {\tt Guy.Fayolle@inria.fr}}} 
\date{December 2008}

\maketitle

\begin{abstract} We investigate the almost sure asymptotic properties of vector martingale transforms. Assuming some appropriate regularity conditions both on the increasing process and on the moments of the martingale, we prove that normalized moments of any even order converge in the almost sure cental limit theorem for martingales. A conjecture about almost sure upper bounds under wider hypotheses is formulated. The theoretical results are supported by examples borrowed from statistical applications, including linear autoregressive models and branching processes with immigration, for which new asymptotic properties are established on estimation and prediction errors.
\end{abstract}

\newpage
\fi
\section{Introduction}
Let $(X_{n})$ be a sequence of real independent identically distributed random variables with $\E[X_{n}]=0$ and $\E[X_{n}^{2}]=\sigma ^{2}$ and denote
$$S_{n}= \sum_{k=1}^n  X_k.$$
It follows from the ordinary central limit theorem (CLT) that 
$$
\frac{S_n}{\sqrt{n}}\liml \CN(0,\sigma^2),
$$
which implies, for any bounded continuous real function $h$
$$
\lim_{n \to \infty}
\E\Bigl[h\Bigl(\frac{S_{n}}{\sqrt{n}}\Bigr)\Bigr]=\int_{\Rset{}}h(x)dG(x)
$$
where $G$ stands for the Gaussian measure $\CN(0,\sigma ^{2})$. 
Moreover, by the celebrated almost sure central limit theorem (ASCLT), the empirical measure 
$$
G_n = \frac{1}{\log n}\sum_{k=1}^{n}\frac{1}{k}\delta_{\frac{S_{k}}{\sqrt{k}}}
$$
satisfies
$$
G_n \Longrightarrow G \hspace{1cm} \hbox{a.s.}
\vspace{2ex}
$$ 
In other words, for any bounded continuous real function $h$
$$
\lim_{n \to \infty}\frac{1}{\log n}\sum_{k=1}^{n}\frac{1}{k}h\Bigl(\frac{S_{k}}{\sqrt{k}}\Bigr) = \int_{\Rset{}}h(x)dG(x)
\hspace{1cm} \hbox{a.s.}
$$ 
The ASCLT was simultaneously proved by Brosamler \cite{Brosamler} and Schatte \cite{Schatte88} and, in its present form, 
by Lacey and Phillip \cite{LP}. In contrast with the wide literature on the ASCLT for independent random variables, very 
few references are available on the ASCLT for martingales apart the recent work of Bercu and Fort \cite{Bercu, BercuFort}
and the important contribution of Chaabane and Maaouia \cite{CM, Chaabane01} and 
Lifshits \cite{Lifshits01,Lifshits02}. The ASCLT for martingales is as follows.
Let $(\ep_{n})$ be a martingale difference sequence adapted to a
filtration $\F= (\CF_{n})$ with $\E[\ep_{n+1}^2|\CF_{n}]=\sigma ^2$ a.s. 
Let $(\Phi_{n})$ be a sequence of random variables
adapted to  $\F$ and denote by $(M_{n})$ the real martingale transform 
$$
M_{n}= \sum_{k=1}^{n}\Phi_{k-1}\ep_{k}.
$$
We also need to introduce the explosion coefficient  $f_{n}$ associated with 
$(\Phi_{n})$
$$
f_{n}=\frac{\Phi_{n}^2}{s_{n}} \hspace{1cm} \mbox{where} \hspace{1cm} 
s_{n}=\sum_{k=0}^{n}\Phi_{k}^2.
$$ 
As soon as $(f_n)$ goes to zero a.s. and under reasonable assumption on the conditional
momments of $(\ep_n)$, the ASCLT for martingales asserts that the empirical measure 
\begin{equation}
\label{tlcpsteo}
G_n = \frac{1}{\log s_n}\sum_{k=1}^{n}f_k\delta_{\frac{M_{k}}{\sqrt{s_{k-1}}}}
\Longrightarrow G \hspace{1cm} \hbox{a.s.}
\vspace{2ex}
\end{equation}
In other words, for any bounded continuous real function $h$
\begin{equation}
\label{tlcps}
\lim_{n \to \infty}\frac{1}{\log s_n}
\sum_{k=1}^{n}f_kh\Bigl(\frac{M_k}{\sqrt{s_{k-1}}}\Bigr)= 
\int_{\Rset{}^{}}h(x)dG(x)
\hspace{1cm} \hbox{a.s.}
\end{equation}

It is quite natural to overcome the case of unbounded functions $h$.
To be more precisely, one might wonder if convergence (\ref{tlcps}) remains true for unbounded functions $h$. 
It has been recently shown in  \cite{Bercu, BercuFort} that, whenever $(\ep_{n})$ has a finite conditional moment of order $>2p$ , then the convergence (\ref{tlcps}) still holds for any continuous real function $h$ such that
$| h(x) | \leq x^{2p}$.
\begin{theo} [Convergence of moments in the scalar case]
\label{tlcpsbernard}
Assume that $(\ep_n)$ is a martingale difference sequence such that $\E[\ep_{n+1}^2| \CF_{n}]=\sigma ^2$ a.s. and satisfying 
for some integer $p\geq 1$ and for some real number $a>2p$,
$$
\sup_{n \geq 0}\E[|\ep_{n+1}|^{a}|\CF_{n}] < \infty \hspace{1cm} \hbox{a.s.}
$$
If the explosion coefficient $(f_{n})$ tends to zero a.s., then
\begin{equation}
\label{cvmomentbernard}
\lim_{n \to \infty}\frac{1}{\log
  s_{n}}\sum_{k=1}^{n}f_{k}\Bigl(\frac{M_{k}^{2}}{s_{k-1}}\Bigr)^{p} =
  \frac{\sigma^{2p}(2p)!}{2^p p!} \hspace{1cm} \hbox{a.s.}
\end{equation}
\end{theo}

The limit (\ref{cvmomentbernard}) is exactly the moment of order $2p$ of the Gaussian distribution $\CN(0,\sigma^{2})$. 
The purpose of the present paper is to extend the results of \cite{Bercu} to vector martingale transforms, 
which is strongly needed in various applications arising in statistics and signal processing.

\medskip\noindent
Let $(M_{n})$ be a square integrable vector martingale with values in $\Rset{}^d$, adapted to a filtration $\F$. 
Its increasing process consists of the sequence $(\langle M \rangle _{n})$ of symmetric, positive semi-definite 
square matrices of order $d$ given by
\[\langle M \rangle _{n}=\sum_{k=1}^n
\E[(M_{k}-M_{k-1})(M_{k}-M_{k-1})^t|\CF_{k-1}].
\]
A first version of the ASCLT for discrete vector martingales was proposed in
\cite{CM, Chaabane98}, under fairly restrictive assumptions on the increasing process $(\langle M \rangle _{n})$. 
Hereafter, our goal is to establish the convergence of moments of even order in the ASCLT for $(M_{n})$ 
under suitable assumptions on the behaviour of $(\langle M \rangle _{n})$.
We shall work in the general setting of vector martingales transforms $(M_{n})$, which can be written as
$$
M_{n}=M_{0}+\sum_{k=1}^{n} \Phi_{k-1}\ep_{k}
$$
where $M_{0}$ can be taken arbitrary and $(\Phi_{n})$ denotes a sequence of random 
vectors of dimension $d$, adapted to $\F$.
We also introduce
\begin{equation} 
\label{Sn}
S_{n}=\sum_{k=0}^{n} \Phi_{k}\Phi_{k}^{t} +S
\end{equation}
where $S$ is a fixed deterministic matrix, symmetric and positive definite.
One can obviously see that if
$\E[\ep_{n+1}^2|\CF_n]=\sigma^2$ a.s., then the increasing process of  
$(M_{n})$ takes the form $\langle M \rangle_{n}=\sigma^{2}S_{n-1}$. 
The explosion coefficient associated with $(\Phi_{n})$ is now given by
\begin{equation} 
\label{def}
f_{n}= \Phi_{n}^{t}S_{n}^{-1}\Phi_{n}=\frac{d_n - d_{n-1}}{d_n}
\end{equation}
where $d_n=\det(S_n)$. 
\bigskip

After this short survey, the paper will be organized as follows.
The main theoretical result for vector martingale transforms is given in 
Section~\ref{sec:cv-moments}, at the end of which a quite plausible interesting conjecture is formulated, involving minimal assumptions. The final Section~\ref{ssec:martingalesapplications} proposes some statistical applications to estimation and prediction errors in linear  autoregressive models and in branching processes with immigration.

\section{On the convergence of moments}
\label{sec:cv-moments}
As mentioned above, our main result is given in theorem \ref{thm2} and extends  theorem  \ref{tlcpsbernard} to the vector case. By the way, in mathematics, the difficulty of the problem is almost always a strictly increasing function of the dimension of some underlying  space: it is also the case here !

\begin{theo}
\label{thm2}
Let $(\ep_n)$ be a martingale difference sequence satisfying the homogeneity condition $\E[\ep_{n+1}^2| \CF_{n}]=\sigma ^2$ a.s. and such that, 
for some integer $p\geq 1$ and some real number $a>2p$,
$$
\begin{array}{lll}
{\displaystyle \sup_{n \geq 0}}\,
\E\bigl[|\ep_{n+1}|^{a}\big|\CF_{n}\bigr]< \infty \hspace{1cm}\mbox{a.s.}
\end{array}
\leqno {(H_p)}
$$
In addition, assume that the explosion coefficient $f_{n}$ tends to zero a.s.
and that there exists a positive random sequence $(\alpha_n)$ increasing to infinity and
an invertible symmetric matrix $L$, such that
\begin{equation}
\label{convergence}
\lim_{n \to \infty}\frac{1}{\alpha_{n}}S_{n}=L \hspace{1cm}\mbox{a.s.}
\end{equation}
Then, the following limits hold almost surely
\begin{eqnarray}
\lim_{n\to \infty} \frac{1}{\log
  d_{n}}\sum_{k=1}^{n}f_{k}\bigl(M_{k}^{t}S_{k-1}^{-1}M_{k}\bigr)^{p}
  = \ell(p) = 
  d\sigma^{2p}\prod_{j=1}^{p-1}\bigl(d+2j\bigr),\label{thm2-1}\\
  \lim_{n\to \infty} \frac{1}{\log
  d_{n}}\sum_{k=1}^{n}\bigl(M_{k}^{t}S_{k-1}^{-1}M_{k}\bigr)^{p}-
  \bigl(M_{k}^{t}S_{k}^{-1}M_{k}\bigr)^{p} = \lambda(p) =
  \frac{p}{d}\ell(p). \label{thm2-2}
\end{eqnarray}
\end{theo}

\begin{rem}
The limit $\ell(p)$ corresponds exactly to the moment of order $2p$ of the norm of a gaussian 
vector $\CN(0,\sigma^2 I_d)$, so that theorem~\ref{thm2} shows indeed the convergence of moments of order $2p$ in the ASCLT for  vector martingales. 
Here the deterministic normalization given in \cite{CM} has been replaced by the  
natural random normalization given by the increasing process.
\end{rem}

\begin{rem}
The convergence hypothesis (\ref{convergence}) clearly implies $f_{n}\to 0$  a.s., if and only if
$\alpha_{n}\sim\alpha_{n-1}$ a.s. As a matter of fact, we deduce from (\ref{convergence}) that
$$ \lim_{n \to \infty}\frac{d_n}{\alpha_{n}^{d}}=\det L>0 \hspace{1cm}\mbox{a.s.}$$
\end{rem}

\begin{proof}
For the sake of shortness, we shall define the following variables.
\begin{eqnarray}
V_{n}&=& M_{n}^{t}S_{n-1}^{-1}M_{n},\label{Vn}\\
\varphi_{n}&=&\alpha_{n}^{-1}\Phi_{n}^{t}L^{-1}\Phi_{n}, \nonumber\\
v_{n}&=&\alpha_{n-1}^{-1}M_{n}^{t}L^{-1}M_{n}. \nonumber
\end{eqnarray}

First of all, by using the symmetry of  $L$,  the 
convergence (\ref{convergence}) ensures that 
\begin{eqnarray}
f_{n}&=&\varphi_{n}+o\bigl(\varphi_{n}\bigr) \quad \mbox{a.s.} \label{simplif}\\
V_{n}&=&v_{n}+o\bigl(v_{n}\bigr)
\quad \mbox{a.s.} \label{simpliV}
\end{eqnarray}
For we have
\[
f_{n}=\varphi_{n}+\alpha_{n}^{-1}\Phi_{n}^{t}L^{-1/2}\bigl(\alpha_{n}L^{1/2}S_{n}^{-1}L^{1/2}-I\bigr)L^{-1/2}\Phi_{n}.
\]
The matrix $R_{n}=\alpha_{n}L^{1/2}S_{n}^{-1}L^{1/2}-I$ is symmetricand denoting by $\rho_{n}$ its spectral radius, we can write
\[\Bigl|\alpha_{n}^{-1}\Phi_{n}^{t}L^{-1/2}R_{n}L^{-1/2}\Phi_{n}\Bigr|
\leq \rho_{n} \ \varphi_{n},\]
and $\rho_{n}$ converges to $0$ almost surely, which leads to (\ref{simplif}). The equation (\ref{simpliV}) is proved in the same way from the decomposition
\[V_{n}=v_{n}+M_{n}^{t}\bigl(S_{n-1}^{-1}-\alpha_{n-1}^{-1}L^{-1}\bigr)M_{n}.\]
Hence, by (\ref{simpliV}),  $V_{n}^{p}=v_{n}^{p} + o\bigl(v_{n}^{p}\bigr)$ a.s. 
In order to find the limit (\ref{thm2-1}), il suffices, by Toeplitz lemma, to study the convergence
\begin{equation}
\label{raccourci}
\lim_{n \to \infty}\frac{1}{\log d_{n}}\sum_{k=1}^{n}f_{k}V_{k}^{p}=\lim_{n \to
  \infty}\frac{1}{\log d_{n}}\sum_{k=1}^{n}\varphi_{k}v_{k}^{p}.
\end{equation}
Theorem~\ref{thm2} will be proved by induction with respect to $p\geq 1$, as  
in  \cite{Bercu} in the scalar case. The first step consists in writing  a recursive relation for $M_{n}^{t}L^{-1}M_{n}$.

Let
\begin{equation} \label{defbeta}
\begin{split}
\beta_{n} &=\tr\bigl(L^{-1/2}S_nL^{-1/2}\bigr),  \quad 
\gamma_{n} = \frac{\beta_{n}-\beta_{n-1}}{\beta_{n}}, \\
 \delta_{n} & =\frac{M_{n}^{t}L^{-1}\Phi_{n}}{\beta_{n}} , \hspace{1.5cm}
m_{n} = \beta_{n-1}^{-1}M_{n}^{t}L^{-1}M_{n}.
\end{split}
\end{equation}
According to the definition of  $(M_{n})$, the following decomposition holds
\begin{equation*}
M_{n+1}^{t}L^{-1}M_{n+1}=M_{n}^{t}L^{-1}M_{n} +
2\ep_{n+1}\Phi_{n}^{t}L^{-1}M_{n}+\ep_{n+1}^{2}\Phi_{n}^{t}L^{-1}\Phi_{n},
\end{equation*}
so that
\begin{equation} \label{decompMn}
m_{n+1}=(1-\gamma_n)m_n+ 2\delta_n\ep_{n+1}+\gamma_n\ep_{n+1}^{2}.
\end{equation}
The theorem relies essentially on the following lemma.
\begin{lem}\label{lemme_martingales}
Under the assumptions of theorem~\ref{thm2}, we have  
\begin{equation}\label{lemme-Vn}
\lim_{n\to \infty} \frac{1}{\log
  d_{n}}\sum_{k=1}^{n}\gamma_{k}m_{k}^{p}=\frac{\ell(p)}{d^{p+1}}\hspace{1cm}\mbox{a.s.}
\end{equation}
In addition, if $g_n=M_n^tS_{n-1}^{-1}\Phi_n$, we also have
\begin{equation}\label{raccourci-an}
\lim_{n\to \infty}\frac{1}{\log
  d_{n}}\sum_{k=1}^{n}(1-f_k)g_k^2m_{k}^{p-1}=\frac{\lambda(p)}{pd^{p-1}} \hspace{1cm}\mbox{a.s.}
\end{equation}
\end{lem}
\begin{proof}
Raising equality (\ref{decompMn}) to the power $p$ implies 
\begin{equation} \label{decompositioncasp}
m_{n+1}^{p}=\sum_{k=0}^{p}\sum_{\ell=0}^{k}2^{k-\ell}C_{p}^{k}C_{k}^{\ell}\
\gamma_{n}^{\ell}\delta_{n}^{k-\ell}\bigl((1-\gamma_n)m_{n}
\bigr)^{p-k}\ep_{n+1}^{k+\ell}.
\end{equation}

After some straightforward simplifications, we obtain the relation
\begin{equation}
\label{decompclep}
m_{n+1}^{p}+\CA_{n}(p)=m_{1}^{p}+ \CB_{n+1}(p)+\CW_{n+1}(p),
\end{equation}
where
\[
\CA_{n}(p) =
\sum_{k=1}^{n}\beta_{k}^{-p}\bigl(\beta_{k}^{p}-\beta_{k-1}^{p}\bigr)m_{k}^{p},
\quad  
\CB_{n+1}(p) = \sum_{\ell=1}^{2p-1} \sum_{k=1}^{n}b_{k}(\ell)\ep_{k+1}^{\ell},
\] 
\[\CW_{n+1}(p) = \sum_{k=1}^{n}\gamma_{k}^{p}\ep_{k+1}^{2p}. 
\]
For $1 \leq \ell \leq p-1$, we have
\[
b_{k}(\ell)= \sum_{j=0}^{\lfloor \ell/2 \rfloor}2^{\ell
  -2j}C_{p}^{\ell-j}C_{\ell -j}^{j}\gamma_{k}^{j}\delta_{k}^{\ell
-2j}\bigl((1-\gamma_{k})m_{k}\bigr)^{p-\ell+j},
\] 
while, for $p \leq \ell \leq 2p-1$, 
\[
\begin{split}
b_{k}(\ell) & = \sum_{j=\ell-(p-1)}^{\lfloor \ell/2 \rfloor}2^{\ell
  -2j}C_{p}^{\ell-j}C_{\ell -j}^{j}\gamma_{k}^{j}\delta_{k}^{\ell
  -2j}\bigl((1-\gamma_{k})m_{k}\bigr)^{p-\ell+j} \\[0.2cm]
&+C_{p}^{\ell-p}2^{2p-\ell}-\delta_{k}^{2p-\ell}\gamma_{k}^{\ell-p}. 
\end{split}
\]
In order to take out useful information about $\CA_{n}(p)$, it is necessary to  study the  asymptotic behaviour  $\CW_{n+1}(p), \CB_{n+1}(p)$ and  $m_{n}^{p}$.
 
\paragraph {The case $p=1.$} Remarking that $\log \beta_n \sim \sum_{k=1}^n\gamma_k$, Chow's lemma  \cite[page \ 22]{Marie} implies
\[
\lim_{n \to \infty}\frac{1}{\log \beta_{n}}\CW_{n+1}(1)=\sigma^2 \hspace{1cm}\mbox{a.s.}
\]
Applying the strong law of large numbers for martingales and the  inequality $\delta_n^{2}\leq \gamma_n m_n$, we get  $\CB_{n+1}(1)=o\bigl(\CA_n(1)\bigr)$ a.s. In addition, from relation  (2.30) in \cite{Wei},  it follows that $m_{n+1}=o(\log \beta_n)$ a.s. 
Consequently, by (\ref{decompclep}), 
\[\lim_{n \to \infty}\frac{1}{\log \beta_{n}}\CA_{n}(1)=\sigma^2 \hspace{1cm}\mbox{a.s.}\]
But the basic convergence assumption (\ref{convergence}) implies immediately
\[
\lim_{n \to \infty}\frac{\beta_{n}}{\gamma_{n}}=d \hspace{1cm}\mbox{a.s.},
\]
so that $\log d_n \sim d \log \beta_n$ a.s.
Hence
\[
\lim_{n \to \infty}\frac{1}{\log d_{n}}\CA_{n}(1)=\frac{\sigma^2}{d} \hspace{1cm}\mbox{a.s.,}
\]
which establishes (\ref{lemme-Vn}). 

As for the proof of (\ref{raccourci-an}), one can proceed in the same way, starting from the decomposition
\begin{eqnarray*}
V_{n+1}&=&M_{n}^{t}S_{n}^{-1}M_{n} +
2\ep_{n+1}\Phi_{n}^{t}S_{n}^{-1}M_{n}+\ep_{n+1}^{2}f_{n}\\
&=& h_{n} +2\ep_{n+1}g_{n}+\ep_{n+1}^{2}f_{n}.
\end{eqnarray*}
which, for all $n\geq 1$, leads to
\begin{equation}\label{eq:decomp}
V_{n+1}+A_{n}=V_{1}+B_{n+1}+W_{n+1},
\end{equation}
where
\[
A_{n} \egaldef \sum_{k=1}^{n}a_{k}(1), \quad
B_{n+1}\egaldef  2\sum_{k=1}^{n}\ep_{k+1}g_{k}, \quad W_{n+1}\egaldef
\sum_{k=1}^{n}\ep_{k+1}^{2}f_{k}.
\]

By Riccati's formula,
\begin{equation} \label{gk}
g_{n}^{2}=(1-f_{n})a_{n}(1),
\end{equation}
and, hence, coupling (\ref{gk}) with the strong law of large numbers for martingales,
\[B_{n+1}=o\bigl(A_{n}\bigr) \quad \mbox{a.s.}
\]
On the other hand, $m_{n}=o\bigl(\log \beta_{n}\bigr)$, which, with (\ref{simpliV}), gives the almost sure estimate $V_{n}=o\bigl(\log d_{n}\bigr)$.

Now, since $\sum_{k=1}^{n}f_{k} \sim \log d_{n}$, Chow's lemma implies 
\[\lim_{n \to \infty}\bigl(\log d_{n}\bigr)^{-1}W_{n+1}=\sigma^{2} \quad \mbox{a.s.},
\] 
and to conclude the proof of  Lemma~\ref{lemme_martingales} for $p=1$, it suffices to divide (\ref{eq:decomp}) by $\log d_{n}$, letting $n\to\infty$.

\paragraph{The case $p \ge 2$.} First, using again Chow's lemma, we can write
\[\CW_{n+1}(p)=o\bigl(\log d_{n}\bigr)\hspace{1cm}\mbox{a.s.}
\]
 Also, we shall show that
\begin{equation} \label{preuvepourB}
\CB_{n+1}(p)=\frac{p}{d^{p+1}}\ell(p)\log d_{n}+ o\bigl(\log d_{n}\bigr)+o\bigl(\CA_n(p)\bigr) \hspace{0.5cm}\mbox{a.s.}
\end{equation}
 Setting, for $1\leq \ell \leq 2p-1$,
\begin{equation}
\label{decompbruit}
\ep_{n+1}^{\ell}\egaldef e_{n+1}(\ell)+\E[\ep_{n+1}^{\ell}|\CF_{n}]=
e_{n+1}(\ell)+\sigma_{n}(\ell),
\end{equation}
we write $\sum_{k=1}^{n}b_{k}(\ell)\ep_{k+1}^{\ell}= \CC_{n+1}(\ell)+\CD_{n}(\ell)$, with
\[\CC_{n+1}(\ell)\egaldef \sum_{k=1}^{n}b_{k}(\ell)e_{k+1}(\ell), \quad \mbox{and}
\quad
\CD_{n}(\ell)\egaldef \sum_{k=1}^{n}b_{k}(\ell)\sigma_{k}(\ell).\]
First, for any $\ell$ such that $1\leq \ell \leq p-1$, using again the strong law of large numbers and equation (2.30) of \cite{Wei}, we have
\[\CC_{n+1}(\ell)=o\bigl(\log d_{n} \bigr) \quad \mbox{a.s.}\]
Suppose $3 \leq \ell\leq p-1$. From Hölder's inequality and the assumptions on the moments of $\ep_n$, it follows that, for all $1\leq j \leq 2p-1$, $ |\sigma_{n}(j)|$ is bounded, which implies
\[\bigl|\CD_{n}(\ell)\bigr|=\CO\Bigl(\sum_{k=1}^{n}\gamma_{k}^{\ell/2}m_{k}^{p-\ell/2}\Bigr).\]
For even $\ell$, the induction assumption leads to $\CD_{n}(\ell)=o(\log d_{n})$ a.s. When $\ell$ is odd, Cauchy Schwarz inequality and the induction assumption yield 
\[
|\CD_{n}(\ell)|=\CO
\Bigl(\bigl(\sum_{k=1}^{n}\gamma_{k}m_{k}^{p-1}\bigr)^{1/2}\bigl(\sum_{k=1}^{n}\gamma_{k}^{\ell}m_{k}^{p-\ell}\bigr)^{1/2}\Bigr)=o\bigl(\log
d_{n}\bigr)\quad \mbox{a.s.}
\]
Suppose now $p \leq \ell \leq 2p-1$. It is easy to obtain
\[|\CD_{n}(\ell)|=\CO\Bigl(\sum_{k=1}^{n}\gamma_{k}^{\ell/2}m_{k}^{p-\ell/2}\Bigr)
\quad \mbox{a.s.}\]
Now, from the induction assumption, we get, for any integer $\ell \neq 2$, 
\[\CD_{n}(\ell)=o\bigl(\log d_{n}\bigr) \quad \mbox{a.s.}\]
It remains to study $\CC_{n+1}(\ell)$. By Chow's lemma, we have the  almost sure equality
\[
\CC_{n+1}(\ell)=o\bigl(\nu_{n}(\ell)\bigr), \quad \mbox{with} \quad \nu_{n}(\ell)=\sum_{k=1}^{n}|b_{k}(\ell)|^{2p/\ell}=\CO\Bigl(\sum_{k=1}^{n}\gamma_{k}^{p}m_{k}^{(2p-\ell)p/\ell}\Bigr).
\]
 For $\ell>p$, we apply Hölder's inequality with exponents $\ell/p$ and $\ell/(\ell-p)$. Then, $\nu_{n}(\ell)=o(\log d_n)$ a.s. In the particular case $p=\ell$, we get by the strong law of large numbers
\[
|\CC_{n+1}(\ell)|^2=\CO\bigl(\tau_{n}(p)\log\tau_{n}(p) \bigr) \quad \mbox{with} \quad \tau_{n}(p)= \sum_{k=1}^{n}b_{k}(p)^2=\CO\Bigl(\sum_{k=1}^{n}\gamma_{k}^{p}m_{k}^{p}\Bigr),
\]
which leads to $|\CC_{n+1}(\ell)|=o(\log d_n)$ a.s., since, from equation  (2.30) of \cite{Wei}, 
$m_k^p=o\bigl((\log d_n)^\delta\bigr)$, for $0<\delta<1$.
As for the last term $\CD_{n}(2)$, one needs to study, for $p\geq 3$, the quantity
\begin{eqnarray*}
\frac{\CD_{n}(2)}{\log d_n}&=&\frac{\sigma^2}{\log
  d_n}\sum_{k=1}^{n}\sum_{j=0}^{1}2^{2-2j}C_p^{2-j}C_{2-j}^{j}\gamma_{k}^j\delta_{k}^{2-2j}m_k^{p-2+j},\\
&=&\frac{2p(p-1)\sigma^2}{\log
  d_n}\sum_{k=1}^{n}\delta_{k}^{2}m_k^{p-2}+ \frac{p\sigma^2}{\log
  d_n}\sum_{k=1}^{n}\gamma_{k}m_k^{p-1}.
\end{eqnarray*}
It is easy to establish
\begin{equation}
\label{psik}
g_{n}^{2}=\bigl(\Phi_{n}^{t}S_{n-1}^{-1}M_{n}\bigr)^{2}=d^2 \ \delta_n^2+o\bigl(\gamma_{n}m_{n}\bigr).
\end{equation}
Then, the induction assumption and Toeplitz lemma imply
\[\lim_{n \to \infty} \frac{1}{\log
  d_{n}}\sum_{k=1}^{n}\delta_{k}^{2}m_{k}^{p-2}=\lim_{n \to \infty}
\frac{1}{d^{2}\log d_{n}}\sum_{k=1}^{n}a_{k}(1)m_{k}^{p-2}=\frac{\lambda(p-1)}{(p-1)d^{p}}\quad \mbox{a.s.}
\]
Thus, we obtain
\[
\lim_{n \to \infty} \frac{\CD_{n}(2)}{\log
  d_n}=\ell(p-1)\Bigl(\frac{2p(p-1)\sigma^2}{d^{p+1}}+\frac{p\sigma^2}{d^p}\Bigr)=\frac{p}{d^{p+1}}\ell(p) \quad \mbox{a.s.},
  \]
[this formula being also valid for $p=2$] which proves (\ref{preuvepourB}).

On the other hand, still applying  equation (2.30) of \cite{Wei}, we derive the estimate $m_{n}^{p}=o(\log d_n)$. Thus,
\[\lim_{n \to \infty}\frac{1}{\log d_{n}}\CA_{n}(p)=\frac{p}{d^{p+1}}\ell(p) \hspace{1cm}\mbox{a.s.}
\]
Since 
\begin{equation}
\label{equivp}
\frac{\beta_{n}^{p}-\beta_{n-1}^{p}}{\beta_{n}^{p}}=\gamma_{n}\sum_{q=0}^{p-1}\Bigl(\frac{\beta_{n-1}}{\beta_{n}}\Bigr)^{p-1-q}\sim
p \gamma_{n} \quad \mbox{a.s.},
\end{equation}
we get finally 
\[\lim_{n \to \infty}\frac{1}{\log
  d_{n}}\sum_{k=1}^{n}\gamma_{k}m_{k}^{p}=\lim_{n \to \infty}\frac{1}{p\log
  d_{n}}\CA_{n}(p)=\frac{\ell(p)}{d^{p+1}},\]
and the proof of (\ref{lemme-Vn}) is terminated.

As for the second part of  Lemma~\ref{lemme_martingales}, i.e. equation 
(\ref{raccourci-an}), we could proceed along similar lines,  via the equality
\[
V_{n+1}^{p}=\sum_{k=0}^{p}\sum_{\ell=0}^{k}2^{k-\ell}C_{p}^{k}C_{k}^{\ell}\
f_{n}^{\ell}g_{n}^{k-\ell}h_{n}^{p-k}\ep_{n+1}^{k+\ell} \,,
\]
with $g_n=\Phi_{n}^{t}S_{n-1}^{-1}M_{n}$ and $h_{n}=M_{n}^{t}S_{n}^{-1}M_{n}$.
\end{proof}
The  proof of theorem \ref{thm2} is completed as relations (\ref{lemme-Vn}) and (\ref{raccourci-an}) are clearly direct consequences of (\ref{thm2-1}) and (\ref{thm2-2}), respectively. Indeed, since $\beta_{n} \sim d \alpha_{n}$, we have almost surely  
\[V_{n}\sim d m_{n}, \quad \mbox{and} \quad f_{n}\sim d  \gamma_{n},
\]
hence convergence (\ref{lemme-Vn}) immediately yields (\ref{thm2-1}). Moreover in the particular case $p=1$, the second convergence~(\ref{raccourci-an}) is exactly (\ref{thm2-2}). Now, for $p\geq 2$, the elementary expansion
\begin{equation}
\label{comppuissance}
x^{p}-y^{p}=(x-y)x^{p-1}\sum_{q=0}^{p-1}\Bigl(\frac{y}{x}\Bigr)^{p-1-q}
\end{equation}
leads to 
\begin{equation}\label{equivap}
\begin{split}
a_{n}(p) & = \bigl(M_{n}^{t}S_{n-1}^{-1}M_{n}\bigr)^{p}-\bigl(M_{n}^{t}S_{n}^{-1}M_{n}\bigr)^{p} \\[0.2cm] 
& = a_{n}(1)V_{n}^{p-1}\sum_{q=0}^{p-1}\Bigl(\frac{V_{n}-a_{n}(1)}{V_{n}}\Bigr)^{p-1-q}.
\end{split}
\end{equation}
Riccati's formula yields
\begin{eqnarray*}
a_{n}(1)&=&
(1-f_n)M_{n}^{t}S_{n-1}^{-1}\Phi_{n}\Phi_{n}^{t}S_{n-1}^{-1}M_{n} \\
&\leq& (1-f_{n})\tr(S_{n-1}^{-1/2}\Phi_{n}\Phi_{n}^{t}S_{n-1}^{-1/2})V_{n} \\
&\leq& f_{n}V_{n},\label{fnVnan}
\end{eqnarray*}
hence $a_{n}(1)=o(V_{n})$ and, by (\ref{equivap}),
 \[
 a_{n}(p)\sim p
a_{n}(1)V_{n}^{p-1} \quad \mathrm{a.s.},
\] 
so that finally
\begin{equation}
\label{akp}
\lim_{n \to \infty}\frac{1}{\log
  d_{n}}\sum_{k=1}^{n}a_{k}(p)=\lim_{n \to \infty}\frac{pd^{p-1}}{\log
  d_{n}}\sum_{k=1}^{n}a_{k}(1)m_{k}^{p-1}=\lambda(p) \quad \mbox{a.s.}\end{equation}
\end{proof}
In most of statistical applications encountered so far, the convergence assumption (\ref{convergence}) is satisfied. However, this technical hypothesis somehow  circumvents the vector problem, which in its full generality is not yet solved. Indeed, (\ref{convergence}) entails that all eigenvalues of the matrix $S_n$ grow to infinity at the same speed $\alpha_n$. Thus, our method of proof as some features in common with the scalar case. Hopefully, we should be able to establish the following result, stated for the moment as a conjecture, without assuming (\ref{convergence}).

\begin{conj}
Let $(\ep_n)$ be a martingale difference sequence satisfying the homogeneity condition
 $\E[\ep_{n+1}^2| \CF_{n}]=\sigma ^2$ a.s. and 
assumption $(H_{p})$ introduced in theorem \ref {thm2}, for some integer $p\geq 1$. 
Then, we have
\begin{eqnarray*}
&&\sum_{k=1}^{n}f_{k}\bigl(M_{k}^{t}S_{k-1}^{-1}M_{k}\bigr)^{p}=\CO\bigl(\log
d_n\bigr)\hspace{1cm}\mbox{a.s.} \\
&&\sum_{k=1}^{n}\bigl(M_{k}^{t}S_{k-1}^{-1}M_{k}\bigr)^{p}-
\bigl(M_{k}^{t}S_{k}^{-1}M_{k}\bigr)^{p}=\CO\bigl(\log
d_n\bigr)\hspace{1cm}\mbox{a.s.}
\end{eqnarray*}
\end{conj}
\section{Applications}
\label{ssec:martingalesapplications}
\subsection{Linear regression models}
Theorem~\ref{thm2} is the keystone to understand the asymptotic behavior of cumulative prediction and estimation errors associated to the stochastic regression process given by, for all $n\geq 1$,
\begin{equation}
\label{modele}
X_{n+1}=\theta^{t}\Phi_{n}+\ep_{n+1},
\end{equation}
where $\theta \in \Rset{}^{d}$ is the unknown parameter. The random variables 
$X_{n},\Phi_{n},\ep_{n}$ are the scalar observation, the regression vector and the scalar driven noise, respectively. 
We propose here two applications. The first one concerns stable autoregressive processes
while the second one deals with branching processes with immigration.

\medskip
For a reasonable sequence $(\widehat{\theta}_{n})$ of estimators of $\theta$, we shall investigate the asymptotic performance of $\widehat{\theta}_{n}^{t}\Phi_{n}$, as a predictor of $X_{n+1}$. More precisely, we shall focus on the prediction error $X_{n+1}-\widehat{\theta}_{n}^{t}\Phi_{n}$ and on the estimation error $\widehat{\theta}_{n}-\theta $. In fact, it is more relevant and efficient \cite{GS}  to consider the cumulative prediction and estimation errors defined, for some $p \geq 1$,
as
\begin{eqnarray}
\label{pred}
C_{n}(p)=\sum_{k=0}^{n-1} (X_{k+1}-\widehat{\theta}_{k}^{t}\Phi_{k})^{2p}
\end{eqnarray}
and 
\begin{eqnarray}
\label{est}
G_{n}(p)=\sum_{k=1}^{n} k^{p-1} \|\widehat{\theta}_{k}-\theta\| ^{2p}.
\end{eqnarray}

In the scalar case $d=1$, under suitable moment conditions, asymptotic estimates on $(C_{n}(p))$ and $(G_{n}(p))$ were established in \cite{Bercu} by means of the standard least squares (LS) estimator
\begin{eqnarray}
\label{ls}
\widehat{\theta}_{n}=S_{n-1}^{-1}\sum_{k=1}^{n}\Phi_{k-1}X_{k}.
\end{eqnarray}
It turns out that theorem~\ref{thm2} allows us to improve the results of \cite{Bercu, BercuFort}. 
Up to our knowledge, only partial results in the particular case $p=1$ have been obtained, namely in \cite{Marie,Wei}, where the authors derived the asymptotics of $(C_{n}(p)) $and $(G_{n}(p))$. For the proofs of the strong consistency of the LS estimator for general linear autoregressive model,  we refer the reader to \cite{Lai-Wei2,Wei3, DST}.  Also, one can find in  \cite{DST,Lai-Wei2,Lai-Wei1,Wei3} various results on the asymptotic behavior of the empirical estimator of the covariance associated with process (\ref{modele}).  

\noindent
One might wonder how the convergence of the moments in the ASCLT helps us to deduce the almost sure
asymptotic properties of the sequences $(C_{n}(p))$ and $(G_{n}(p))$. It follows from
(\ref{modele}) and (\ref{ls}) that
\begin{equation}
\label{diff}
\widehat{\theta}_{n}-\theta=S_{n-1}^{-1}M_{n}
\end{equation}
where 
$$
M_{n}= M_{0}+\sum_{k=1}^{n} \Phi_{k-1}\ep_{k}
$$
with $M_{0}=-S\theta$. If
\begin{equation}
\label{pin}
\pi_{n}=\bigl(\theta-\widehat \theta_{n}\bigr)^{t}\Phi_{n}
=X_{n+1}-\widehat \theta_{n}^{t}\Phi_{n}-\ep_{n+1},
\end{equation}
relations (\ref{diff}) and (\ref{pin}) yield
\[\pi_{n}^{2}=M_{n}^{t}S_{n-1}^{-1}\Phi_{n}\Phi_{n}^{t}S_{n-1}^{-1}M_{n}.\]
Applying now Riccati's formula given e.g. in \cite{Marie} pages 96 and 99, it comes
\[S_{n-1}^{-1}=S_{n}^{-1}+(1-f_{n})
S_{n-1}^{-1}\Phi_{n}\Phi_{n}^{t} S_{n-1}^{-1}.\]
Hence, one can write
\begin{eqnarray}
\label{a}
a_{n}(1)= M_{n}^{t}S_{n-1}^{-1}M_{n}
-M_{n}^{t}S_{n}^{-1}M_{n}=(1-f_{n})\pi_{n}^{2}.
\end{eqnarray}
It is often difficult to obtain asymptotic information about the explosion coefficient $f_{n}$. Nevertheless, in our framework,
it is possible to show that $(f_{n})$ converges almost surely to zero. So, the asymptotic behavior of $(G_{n}(p))$ and $(C_{n}(p))$ can be derived from the properties of $\bigl(a_{n}(1)\bigr)^{p}$, under some suitable moment conditions 
on the driven noise $(\ep_{n})$. By the same token, the moments of order $2p$ can be estimated and controlled. 
\begin{cor}
\label{corollairepouran}
Under the assumptions of theorem~\ref{thm2}, one has almost surely
\begin{equation}
\label{res2bis}
\lim_{n \to \infty}\frac{1}{\log
  d_{n}}\sum_{k=1}^{n}\bigl(a_{k}(1)\bigr)^{p}=\left\{\begin{array}{lll} 0 & \mbox{
  if } & p>1,\\
\sigma^2 & \mbox{
  if } & p=1. \end{array} \right.
\end{equation}
\end{cor}
\begin{proof}
When $p=1$, the convergence (\ref{res2bis}) corresponds precisely to (\ref{thm2-2}). Suppose now that $p>1$. Since
\[
a_n(1)\leq f_nV_n 
\]
and $f_{n}\to 0$ almost surely, we get at once, applying Lemma~\ref{lemme_martingales} and Kronecker's lemma,
\[ 0 \leq \lim_{n \to \infty}\frac{1}{\log
  d_{n}}\sum_{k=1}^{n}\bigl(a_{k}(1)\bigr)^{p} \leq \lim_{n \to
  \infty}\frac{1}{\log
  d_{n}}\sum_{k=1}^{n}\bigl(f_{k}V_{k}\bigr)^{p}=0 \hspace{0.5cm} \hbox{a.s.}\]
\end{proof}
\subsection{Moment estimation, prediction and estimation errors}
\label{ssec:martingalesestimation}
Assume $(\ep_n)$ forms a martingale difference sequence with $\E[\ep_{n+1}^2| \CF_{n}]=\sigma ^2$ a.s., for all $n\ge 1$, and let
\[\Delta_{n}= \frac{1}{n}\sum_{k=1}^{n}\ep_{k}^{2}.\] 
If $\ep_n$ has a conditional moment of order $a>2$, then 
the strong law of large numbers for martingales implies the almost sure convergence of $\Delta_n$  to $\sigma^2$. Under the assumptions of theorem~\ref{thm2}
with $p=1$, the convergence (\ref{res2bis}) leads to the strong consistency of the estimator of $\sigma^{2}$
\[\Gamma_{n}=\frac{1}{n}\sum_{k=0}^{n-1}
\bigl(X_{k+1}-\widehat{\theta}_{k}^{t}\Phi_{k}\bigr)^{2},
\]
since
\begin{eqnarray}
\label{estmom}
\lim_{n \to \infty}\frac{n}{\log
  d_{n}}(\Gamma_{n}-\Delta_{n})=\sigma^{2}\hspace{1cm} \hbox{a.s.}
\end{eqnarray}
Hence, a natural estimator of higher moments of $(\ep_{n})$ can be proposed. For any integer $p \geq 1$, let
\begin{eqnarray}
\label{gamma}
\Gamma_{n}(2q)=\frac{1}{n}\sum_{k=0}^{n-1}
(X_{k+1}-\widehat{\theta}_{k}^{t}\Phi_{k})^{2p}.
\end{eqnarray}
One can readily observe that $n \Gamma_{n}(2p)=C_{n}(p)$. 
For any integer $p \geq 1$, let 
\begin{eqnarray}
\label{delta}
\Delta_{n}(2p)=\frac{1}{n}\sum_{k=0}^{n-1}
\ep_{k}^{2p}.
\end{eqnarray}
Almost sure asymptotic properties of $\Gamma_{n}(2p)$ are given in the next corollary.
\begin{cor}
\label{erreur}
Assume that $(\ep_{n})$ satisfies $(H_{p})$ with $p\geq 1$.
In addition, suppose  that for some integer $1 \leq q \leq p$, 
$\E\bigl[\ep_{n+1}^{2q}\big|\CF_{n}\bigr]=\sigma(2q)$ a.s. Then, 
$\Gamma_{n}(2q)$ is a strongly consistent estimator of $\sigma(2q)$ with
\begin{eqnarray}
\label{estmoment}
\Bigl(\Gamma_{n}(2q)-\Delta_{n}(2q)\Bigr)^{2} =
  \CO\Bigl(\frac{\log n}{n}\Bigr)
\end{eqnarray}
\end{cor}
\begin{proof}
We already saw via (\ref{estmom}) that Corollary~\ref{erreur} holds for $q=1$.  Assume now  $q\geq 2$.  By expanding the expression of $\Gamma_{n}(2q)$, equality (\ref{pin}) leads to
\[n\bigl(\Gamma_{n}(2q)-\Delta_{n}(2q)\bigr) =
\sum_{k=0}^{n-1}\pi_{k}^{2q}
+\sum_{\ell=1}^{2q-1}C_{q}^{\ell}\sum_{k=0}^{n-1} 
\pi_{k}^{2q-\ell}\ep_{k+1}^{\ell}.\]
We deduce from (\ref{a}) together with the almost sure convergence of $f_{n}$ to zero and Corollary~\ref{corollairepouran}
that
\[\sum_{k=0}^{n}\pi_{k}^{2q}=o(\log d_{n}) \hspace{1cm} \hbox{a.s.}\]
For all $\ell \in \{1, \ldots, 2q-1\}$, let us write
\[
\sum_{k=0}^{n-1}\pi_{k}^{2q-\ell}\ep_{k+1}^{\ell}=P_n(\ell)+Q_n(\ell)
\]
with
\[
P_n(\ell)= \sum_{k=0}^{n-1}\pi_{k}^{2q-\ell}\sigma_{k}(\ell)
\qquad \mbox{and} \qquad
Q_n(\ell)= \sum_{k=0}^{n-1}\pi_{k}^{2q-\ell}e_{k+1}(\ell),
\]
where $\sigma_n(\ell)=\E\bigl[\ep_{n+1}^{\ell}\big|\CF_{n}\bigr]$ and 
$e_{n+1}(\ell)=\ep_{n+1}^{\ell}-\sigma_n(\ell)$. First, since the moments $\sigma_n(\ell)$ are almost surely bounded, it comes 
\[
|P_n(\ell)|=\CO(\sum_{k=0}^{n-1}
\pi_{k}^{2q-\ell})=\CO(\log d_{n}) \hspace{1cm} \hbox{a.s.}
\]
Moreover, from the estimate
\[
|Q_n(\ell)|^2=\CO(n \log d_{n}) \hspace{1cm} \hbox{a.s.}
\]
we get
\[n^{2}\bigl(\Gamma_{n}(2q)-\Delta_{n}(2q)\bigr)^{2}=\CO\bigl(n \log d_{n})
\quad \mbox{a.s.},\]
which concludes the proof of Corollary~\ref{erreur}.
\end{proof}
It is now possible to deduce from Corollary~\ref{erreur} the asymptotic behavior of $(C_{n}(q))$. Under the assumptions of Corollary~\ref{erreur}, the
convergence~(\ref{estmoment}) yields $C_{n}(q)/n$ converges a.s. to
$\sigma (2q)$. Moreover since the conditional moment of order $a>2p$ of $(\ep_{n})$ is almost surely finite, Chow Lemma leads to 
\begin{eqnarray}
\label{*}
\left| \frac{1}{n}\sum_{k=1}^{n}\ep_{k}^{2q}-\sigma(2q)\right|
=o(n^{c-1}) \hspace{1cm} \hbox{a.s.}
\end{eqnarray}
for all $c$ such that $2pa^{-1}<c<1$. Hence it follows from (\ref{estmoment}) and (\ref{*}) that, if $\log d_{n}=o\bigl(n^{c}\bigr)$, 
\[\left| \frac{1}{n}\ C_{n}(q)-\sigma(2q)\right|^{2}
=o(n^{c-1})\hspace{1cm} \hbox{a.s.}\]
Before enoncing the result on the cumulative estimation error $(G_{n}(p))$, we need another corollary of theorem~\ref{thm2}.
\begin{cor}
\label{avcv}
Under the assumptions of theorem~\ref{thm2}, one has
\begin{eqnarray}
\label{cvmoy}
\lim_{n \to \infty} \frac{1}{\log d_{n}}\sum_{k=1}^{n}
f_{k}\left((\widehat{\theta_{k}}-\theta)^{t}S_{k}(\widehat{\theta_{k}}-\theta)
\right)^{p}=\ell(p) \hspace{0.5cm} \hbox{a.s.}
\end{eqnarray}
In addition, assume that it exists a positive definite symmetric matrix $L$ such that 
\begin{eqnarray}
\label{cv}
\lim_{n \rightarrow +\infty} \frac{1}{n} \ S_{n}=L \hspace{1cm} \hbox{a.s.}
\end{eqnarray}
Then, we have 
\begin{eqnarray}
\label{res5}
\lim_{n \to \infty} \frac{1}{\log
n}\sum_{k=1}^{n}k^{p-1}\left((\widehat{\theta_{k}}-
\theta)^{t}L(\widehat{\theta_{k}}-\theta)\right)^{p}= \ell(p) \hspace{0.5cm} \hbox{a.s.}
\end{eqnarray}
\end{cor}
\begin{rem}
Since $L$ is positive definite, (\ref{res5}) immediately yields 
\[
G_{n}(p)=\CO\bigl(\log n\bigr) \hspace{1cm} \mbox{a.s.}
\]
\end{rem}

\begin{proof}
From the definitions of $S_{n}$ and of $\widehat \theta_{n}$, it is easy to see that 
\[\bigl(\widehat \theta_{n}-\theta
  \bigr)^{t}S_{n}\bigl(\widehat \theta_{n}-\theta
  \bigr)=V_{n}+g_{n}^{2}.\]
Hence, it follows from the convergence~(\ref{thm2-1}) and from the convergence of the explosion coefficient $(f_{n})$ to zero, together with Kronecker's Lemma that almost surely
\[\lim_{n \to \infty}\frac{1}{\log
  d_{n}}\sum_{k=1}^{n}f_{k}\Bigl(\bigl(\widehat \theta_{k}-\theta
  \bigr)^{t}S_{k}\bigl(\widehat \theta_{k}-\theta
  \bigr)\Bigr)^{p}=\lim_{n \to \infty}\frac{1}{\log
  d_{n}}\sum_{k=1}^{n}f_{k}V_{k}^{p}=\ell(p).\]
Then, the convergence~(\ref{cvmoy}) is a straight forward consequence of theorem~\ref{thm2}. 
Using  the formel object $\sqrt L$, we get 
\[\Bigl(\bigl(\hat
\theta_{n}-\theta\bigr)^{t}L\bigl(\hat
\theta_{n}-\theta\bigr)\Bigr)^{p}\sim \bigl(n^{-2}m_{n}\beta_{n}\bigr)^{p}\sim
\bigl(d^{2}m_{n}\beta_{n}^{-1}\bigr)^{p} \quad \mbox{a.s.}\]
Thus,
\[\lim_{n \to \infty} \frac{1}{\log n} \sum_{k=1}^{n}k^{p-1}\Bigl(\bigl(\hat
\theta_{k}-\theta\bigr)^{t}L\bigl(\hat
\theta_{k}-\theta\bigr)\Bigr)^{p}=\lim_{n \to
\infty}\frac{d^{p+1}}{\log n}\sum_{k=1}^{n}\frac{m_{k}^{p}}{\beta_{k}}
\quad \mbox{a.s.}
\] 
The classical Abel transform gives the decomposition
\begin{equation}
\label{abel}
\sum_{k=1}^{n}\gamma_{k}m_{k}^{p}=\frac{m_{n}^{p}}{\beta_{n}}
\bigl(\Sigma_{n}-d(n-1)\bigr)-\frac{m_{1}^{p}}{\beta_{0}} 
\Sigma_{0}+r_{n}+d\sum_{k=1}^{n-1}\frac{m_{k}^{p}}{\beta_{k}} ,
\end{equation}
with 
\[
\Sigma_{n} = \sum_{k=1}^{n}\beta_{k}\gamma_{k} =
\sum_{k=1}^{n}\Phi_{k}^{t}L^{-1}\Phi_{k} \sim \beta_{n}
\] 
and
\[
r_{n} = \sum_{k=1}^{n-1}\Bigl(\frac{m_{k}^{p}}{\beta_{k}}-\frac{m_{k+1}^{p}}
{\beta_{k+1}}\Bigr)\bigl(\Sigma_{k}-kd\bigl).
\] 
Moreover,
\[\frac{m_{n}^{p}}{\beta_{n}}\bigl(\Sigma_{n}-d(n-1)\bigr)- 
\frac{m_{1}^{p}}{\beta_{0}}\Sigma_{0}=o\bigl(\log d_{n}\bigr) \quad
\mbox{a.s.}\]
So, it only remains to prove that $r_{n}=o\bigl(\log n\bigr)$
a.s. Indeed, lemma~\ref{lemme_martingales} yields
\[\lim_{n \to \infty}\frac{1}{\log
  n}\sum_{k=1}^{n}\frac{m_{k}^{p}}{\beta_{k}}=\frac{1}{d}\lim_{n \to \infty}
  \frac{1}{\log n}\sum_{k=1}^{n}\gamma_{k}m_{k}^{p}=\frac{\ell(p)}{d^{p+1}}
  \quad \mbox{a.s.}\]
Then, splitting $r_{n}$ into two terms
\begin{equation}
\label{rn}
r_{n}=\sum_{k=1}^{n-1}\frac{\Sigma_{k}-kd}{\beta_{k}}
\bigl(m_{k}^{p}-m_{k+1}^{p}\bigr)+\sum_{k=1}^{n-1}
\frac{\Sigma_{k}-kd}{\beta_{k}}\gamma_{k+1}m_{k+1}^{p},\end{equation}
and using the proof of theorem~\ref{thm2} together with (\ref{decompositioncasp}), we obtain
\[
\lim_{n \to \infty}\frac{1}{\log
  n}\sum_{k=1}^{n-1}\frac{\Sigma_{k}-kd}{\beta_{k}}
\Bigl(\beta_{k}^{-p}\bigl(\beta_{k}^{p}-\beta_{k-1}^{p}
\bigr)m_{k}^{p}-w_{k+1}-b_{k+1}\bigl)=0 \quad \mbox{a.s.}
\]
The second term is almost surely $o(\log n)$: this is a mere consequence of Lemma~\ref{lemme_martingales} and  of the a.s. convergence of $\Sigma_{n}-nd/\beta_{n}$ to zero. Finally, we have
\[\lim_{n\to \infty}\frac{1}{\log n} \sum_{k=1}^{n}k^{p-1}\Bigl(\bigl(\hat
\theta_{k}-\theta\bigr)^{t}L\bigl(\hat
\theta_{k}-\theta\bigr)\Bigr)^{p}=\ell(p) \quad \mbox{a.s.}\]
\end{proof}
We shall apply now these asymptotic properties to autoregressive
processes and  to  branching processes with immigration, which are both particular cases of the general stochastic regression process (\ref{modele}).

\subsection{ The linear autoregressive process} 
\label{autoregression}
The linear autoregressive process is defined for all $n \geq 1$ by
\begin{equation}
\label{autoregressiflineaire}
X_{n+1}=\sum_{k=1}^{d}\theta_{k}X_{n-k+1} + \varepsilon_{n+1}.
\end{equation}
Let  $C$ denote the companion matrix associated with $(X_n)$
\begin{displaymath}
\label{matcomp}
C=\left(\begin{array}{ccccc}
        \theta_{1} & \theta_{2} & \ldots & \theta_{d-1} & \theta_{d} \\
        1 & 0  & \ldots & \ldots & 0 \\
        0 & 1  & 0 &\ldots & 0 \\
 \vdots & \ddots & \ddots & \ddots & \vdots \\
  0  & \ldots & 0 & 1 & 0 \\
\end{array}\right).
\end{displaymath}
We shall focus our attention on the stable case which means that we assume that
$\rho(C)<1$ where $\rho(C)$ is the spectral radius of the matrix $C$.
In addition, we also assume that $(\ep_{n})$ is a martingale difference sequence
which satisfies $(H_{p})$ with $p\geq 1$.
If $\E[\ep_{n+1}^2| \CF_{n}]=\sigma ^2$ a.s. and the matrix
\begin{displaymath}
\label{Gamma}
 \Gamma=\sigma^{2} \left(\begin{array}{cccc}
        1 & 0 & \ldots & 0 \\
        0 & 0 & \ldots & 0 \\
       \vdots & \vdots & \vdots & \vdots \\
        0 & 0 & \ldots & 0 \\
\end{array}\right),
\end{displaymath}
then convergence (\ref{cv}) holds with the limiting matrix
$L$ given by
\begin{eqnarray}
\label{L}
L=\sum_{k=0}^{\infty}C^{k}\Gamma(C^{t})^{k}.
\end{eqnarray}
Moreover, one can immediately see that the matrix $L$ is positive definite \cite{Marie, DST}.  Then we are in a position to state our following result.
\begin{cor}
\label{erreurar}
Assume that $(\ep_{n})$ satisfies $(H_{p})$ with $p\geq 1$.
In addition, suppose  that for some integer $1 \leq q \leq p$, 
$\E\bigl[\ep_{n+1}^{2q}\big|\CF_{n}\bigr]=\sigma(2q)$ a.s. Then, 
$\Gamma_{n}(2q)$ is a strongly consistent estimator of $\sigma(2q)$ with
\begin{eqnarray}
\label{estmomentar}
\Bigl(\Gamma_{n}(2q)-\Delta_{n}(2q)\Bigr)^{2} =
  \CO\Bigl(\frac{\log n}{n}\Bigr)\quad \mbox{a.s.}
\end{eqnarray}
Moreover, we also have
\begin{eqnarray}
\label{lfqpar}
\lim_{n \to \infty} \frac{1}{\log
n}\sum_{k=1}^{n}k^{p-1}\left((\widehat{\theta}_{k}-
\theta)^{t}L(\widehat{\theta}_{k}-\theta)\right)^{p}= \ell(p) \quad
\mbox{a.s.}
\end{eqnarray}
\end{cor}

\subsection{A branching process with immigration}

\subsubsection{Estimation of the mean}

Branching process with immigration play an increasingly important role in
statistical physics, computational biology and evolutionary theory.
The concept of immigration is related to situation
in which the population can be enriched by exogenous contributions.
The branching process with immigration $(X_n)$ is given for all $n \geq 1$ by
the recursive relation
\begin{equation}
\label{branchement}
X_{n+1}=\sum_{k=1}^{X_{n}}Y_{n,k}+I_{n+1},
\end{equation}
where $(Y_{n,k})$ and $(I_n)$ are two independent sequences
of i.i.d. nonnegative, 
integer-valued random variables. The initial ancestor is
$X_{0}=1$. The distribution of $(Y_{n,k})$ is commonly
called the offspring distribution, while that of 
of $(I_n)$ is known as the immigration distribution.

Define
\begin{alignat*}{2}
\E[Y_{n,k}] & = m, & \qquad \E[I_{n}]  & = \lambda ,  \\
\var[Y_{n,k}] & = \sigma^{2} , & \qquad  \var[I_{n}]  & = b^{2}.
\end{alignat*}
We are interested in the estimation of all the parameters
$m,\lambda,\sigma^2,b^2$.
Relation (\ref{branchement}) may be rewritten as the autoregressive form
\begin{equation}
\label{branchement2}
X_{n+1}= m X_{n} + \lambda + \ep_{n+1}
\end{equation}
where $\ep_{n+1}=X_{n+1}-mX_n - \lambda$. 
Consequently, the branching process with immigration 
is a particular case of the stochastic regression process given by (\ref{modele}) with $\Phi_{n}^{t}= (X_{n},1)$ and 
$\theta^{t}= (m,\lambda)$.
However, one can observe that the situation is a little bit more tricky
as $(\ep_{n})$ is a martingale difference sequence with unbounded
conditional variance
\[\E\bigl[\ep_{n+1}^{2}| \CF_{n}\bigr]=\sigma^{2}X_{n}+b^{2}.\]
To circumvent this technical difficulty, we introduce the following 
regression process
\[
Z_{n+1} = \theta^{t}\Psi_{n} + \xi_{n+1},
\]
where the random variables $Z_{n+1}$, $\Psi_{n}$ and $\xi_{n}$ are given by
\[
Z_{n+1}= c_{n}^{-1/2}X_{n+1}, \quad \Psi_{n}= c_{n}^{-1/2}\Phi_{n}, \quad
\xi_{n+1}= c_{n}^{-1/2} \ep_{n+1} ,
\]
with
$$c_{n}= X_{n}+1.$$ 
Herafter, $(\xi_{n})$ is a martingale difference sequence with bounded conditional variance
$\E[\xi_{n+1}^{2}| \CF_{n}]\leq \sigma^{2}+b^{2}$ a.s.
The mean vector  $\theta^{t}= (m,\lambda)$ will be estimated by the conditional least-squares estimator
\[
\widehat \theta_{n}= S_{n}^{-1}\sum_{k=1}^{n}c_{k}^{-1}\Phi_{k}X_k ,
\]
where
\[
S_{n}= I_2+\sum_{k=0}^{n}c_{k}^{-1}\Phi_{k}\Phi_{k}^{t}.
\]
In the subcritical case $m<1$, Wei and Winnicki \cite{WeiWinnicki} have shown the almost sure convergence
$$
\lim_{n \to \infty}\frac{1}{n}S_{n}
=L \quad
\mbox{a.s.}
$$
where the limiting matrix $L$ is given by
\begin{displaymath}
\label{Lbranc}
L = \left(\begin{array}{cc}
        \E\Big[\frac{X^{2}}{X+1}\Bigl] & \E\Big[\frac{X}{X+1}\Bigl]\\
                                       &                            \\
        \E\Big[\frac{X}{X+1}\Bigl]     & \E\Big[\frac{1}{X+1}\Bigl]\\
 \end{array}\right).
\end{displaymath}
The notation $X$ stands for a random variable sharing the same distribution as the stationary distribution
of the Markov chain $(X_{n})$. Consequently, the matrix $L$ is positive definite and  the following result holds.
\begin{cor}
\label{erreurbgwi}
Assume that $(\ep_{n})$ satisfies $(H_{p})$ with $p\geq 1$.
In addition, suppose  that for some integer $1 \leq q \leq p$, 
$\E\bigl[\ep_{n+1}^{2q}\big|\CF_{n}\bigr]=\sigma(2q)$ a.s. Then, 
$\Gamma_{n}(2q)$ is a strongly consistent estimator of $\sigma(2q)$ with
\begin{eqnarray}
\label{estmomentbgwi}
\Bigl(\Gamma_{n}(2q)-\Delta_{n}(2q)\Bigr)^{2} =
  \CO\Bigl(\frac{\log n}{n}\Bigr)\quad \mbox{a.s.}
\end{eqnarray}
Moreover, we also have
\begin{eqnarray}
\label{lfqpbgwi}
\lim_{n \to \infty} \frac{1}{\log
n}\sum_{k=1}^{n}k^{p-1}\left((\widehat{\theta}_{k}-
\theta)^{t}L(\widehat{\theta}_{k}-\theta)\right)^{p}= \ell(p) \quad
\mbox{a.s.}
\end{eqnarray}
\end{cor}
\subsubsection{Estimation of the variance}
It follows from equation (\ref{branchement2}) that
\[
\ep_{n+1}^{2}=\sigma^{2}X_{n}+b^{2}+V_{n+1},
\]
where $(V_{n})$ is a martingale difference sequence satisfying 
\[
\E\bigl[V_{n+1}^{2}|\CF_{n}\bigr]=2\sigma^{4}X_{n}^{2}+X_{n}(\tau^{4}-3\sigma^{4}+4b^{2}\sigma^{2})+\nu^{4}-b^{4},
\]
with
$$\tau^{4}=\E[(Y_{n,k}-m)^4] \hspace{0.5cm} \hbox{and} \hspace{0.5cm} \nu^{4}=\E[(I_n-\lambda)^4].$$ 
Consequenly, we infer that
\[\E\bigl[c_{n}^{-2}V_{n+1}^{2}|\CF_{n}\bigr]\leq
\tau^{4}+4b^{2}\sigma^{2}+\nu^{4}.\]
As before, we estimate the vector of variances $\eta^{t}=(\sigma^{2}, b^{2})$ by the following conditional least-squares estimator
\[
\widehat \eta_{n}= Q_{n}^{-1}\sum_{k=1}^{n}c_{k}^{-2}\Phi_{k}\widehat
\ep_{k+1}^{2} \quad \hbox{with} \quad \widehat \ep_{k+1}=X_{k+1}-\widehat \theta_{k}\Phi_{k} ,
\]
where
\[
Q_{n}= I_2+\sum_{k=0}^{n}c_{k}^{-2}\Phi_{k}\Phi_{k}^{t}.
\]
In the subcritical case $m<1$, it was established by Winnicki \cite{Winnicki}
that
\[\lim_{n \to \infty}\frac{1}{n}Q_{n}=\Lambda\]
where $\Lambda$ is the positive definite limiting matrix given by
\begin{displaymath}
\label{Lbrancvar}
\Lambda = \left(\begin{array}{cc}
        \E\Big[\frac{X^{2}}{(X+1)^{2}}\Bigl] &
        \E\Big[\frac{X}{(X+1)^{2}}\Bigl]\\
& \\
        \E\Big[\frac{X}{(X+1)^{2}}\Bigl]    & \E\Big[\frac{1}{(X+1)^{2}}\Bigl]\\
 \end{array}\right).
\end{displaymath}
Our last result is the following:
\begin{cor}
\label{erreureta}
Assume that $(\ep_{n})$ satisfies $(H_{p})$ for  $p\geq 2$.
Then
\begin{eqnarray}
\label{lfqpeta}
\lim_{n \to \infty} \frac{1}{\log
n}\sum_{k=1}^{n}k^{p-1}\left((\widehat{\eta}_{k}-
\eta)^{t}\Lambda(\widehat{\eta}_{k}-\eta)\right)^{p}= \ell(p) \quad
\mbox{a.s.}
\end{eqnarray}
\end{cor}

\bibliographystyle{acm}
\bibliography{Mart-6780}

\begin{thebibliography}{10}

\bibitem{Bercu}
{\sc Bercu, B.}
\newblock {On the convergence of moments in the almost sure central limit
  theorem for martingales with statistical applications}.
\newblock {\em Stochastic Processes and their applications 111\/} (2004),
  157--173.

\bibitem{BercuFort}
{\sc Bercu, B., and Fort, J.~C.}
\newblock {A moment approach for the almost sure central limit theorem for
  martingales}.
\newblock {\em Studia Scientiarum Mathematicarum Hungarica 45}, 1 (2008),
  139--159.

\bibitem{Brosamler}
{\sc Brosamler, G.~A.}
\newblock An almost everywhere central limit theorem.
\newblock {\em Math. Proc. Cambridge Philos. Soc. 104}, 3 (1988), 561--574.

\bibitem{Chaabane01}
{\sc Cha{\^a}bane, F.}
\newblock Invariance principles with logarithmic averaging for martingales.
\newblock {\em Studia Sci. Math. Hungar. 37}, 1-2 (2001), 21--52.

\bibitem{CM}
{\sc Cha{\^a}bane, F., and Ma{\^a}ouia, F.}
\newblock Th\'eor\`emes limites avec poids pour les martingales vectorielles.
\newblock {\em ESAIM Probab. Statist. 4\/} (2000), 137--189.

\bibitem{Chaabane98}
{\sc Cha{\^a}bane, F., Ma{\^a}ouia, F., and Touati, A.}
\newblock G\'en\'eralisation du th\'eor\`eme de la limite centrale
  presque-s\^ur pour les martingales vectorielles.
\newblock {\em C. R. Acad. Sci. Paris S\'er. I Math. 326}, 2 (1998), 229--232.

\bibitem{Marie}
{\sc Duflo, M.}
\newblock {\em Random Iterative Methods}.
\newblock Springer-Verlag, 1997.

\bibitem{DST}
{\sc Duflo, M., Senoussi, R., and Touati, A.}
\newblock {Propri\'{e}t\'{e}s asymptotiques presque s\^{u}res de l'estimateur
  des moindres carr\'{e}s d'un mod\`{e}le autoregressif vectoriel}.
\newblock {\em Ann. Inst. Henri Poincar\'{e} 27}, 1 (1991), 1--25.

\bibitem{GS}
{\sc Goodwin, G., and Sin, K.}
\newblock {\em {Adaptative Filtering Prediction and Control.}}
\newblock Prentice-Hall, Englewood Cliffs, N.J., 1984.

\bibitem{LP}
{\sc Lacey, M., and Phillip, W.}
\newblock {A note on the almost sure central limit theorem}.
\newblock {\em Statist. Probab. Letters 9\/} (1990), 201--205.

\bibitem{Lai-Wei2}
{\sc Lai, T., and Wei, C.}
\newblock {Least-squares estimates in stochastic regression models with
  applications to identification and control of dynamic systems}.
\newblock {\em The Annals of Statistics 10}, 1 (1982), 154--166.

\bibitem{Lai-Wei1}
{\sc Lai, T., and Wei, C.}
\newblock {Asymptotic properties of general autoregressive models and strong
  consistency of least-squares estimates of their parameters}.
\newblock {\em Journal of Multivariate Analysis 13\/} (1983), 1--23.

\bibitem{Lifshits01}
{\sc Lifshits, M.~A.}
\newblock Lecture notes on almost sure limit theorems.
\newblock {\em Publications IRMA 54\/} (2001), 1--25.

\bibitem{Lifshits02}
{\sc Lifshits, M.~A.}
\newblock Almost sure limit theorem for martingales.
\newblock In {\em Limit theorems in probability and statistics}. J\'anos Bolyai
  Math. Soc., Budapest, 2002, pp.~367--390.

\bibitem{Schatte88}
{\sc Schatte, P.}
\newblock On strong versions of the central limit theorem.
\newblock {\em Math. Nachr. 137\/} (1988), 249--256.

\bibitem{Wei3}
{\sc Wei, C.}
\newblock {Asymptotic properties of least-squares estimates in stochastic
  regression models}.
\newblock {\em The Annals of Statistics 13}, 4 (1985), 1498--1508.

\bibitem{Wei}
{\sc Wei, C.}
\newblock {Adaptative prediction by least squares predictors in stochastic
  regression models with applications to time series}.
\newblock {\em The Annals of Statistics 15}, 4 (1987), 1667--1682.

\bibitem{WeiWinnicki}
{\sc Wei, C.~Z., and Winnicki, J.}
\newblock Estimation of the means in the branching process with immigration.
\newblock {\em Ann. Statist. 18}, 4 (1990), 1757--1773.

\bibitem{Winnicki}
{\sc Winnicki, J.}
\newblock Estimation of the variances in the branching process with
  immigration.
\newblock {\em Probab. Theory Rel. Fields 88}, 1 (1991), 77--106.

\end{thebibliography}
\end{document}